\pdfoutput=1
\documentclass[11pt]{amsart}

\usepackage[T1]{fontenc}
\usepackage[utf8]{inputenc}
\usepackage{amsmath,amssymb,amsthm,mathtools}
\usepackage{lmodern}
\usepackage[margin=1in]{geometry}
\usepackage{enumitem}
\usepackage{microtype}
\usepackage[colorlinks=true,linkcolor=blue,citecolor=blue,urlcolor=blue]{hyperref}

\numberwithin{equation}{section}

\theoremstyle{plain}
\newtheorem{theorem}{Theorem}[section]
\newtheorem{proposition}[theorem]{Proposition}
\newtheorem{lemma}[theorem]{Lemma}

\theoremstyle{definition}
\newtheorem{question}[theorem]{Question}

\newcommand{\C}{\mathbb{C}}
\newcommand{\R}{\mathbb{R}}
\newcommand{\Z}{\mathbb{Z}}
\newcommand{\PP}{\mathbb{P}}
\newcommand{\cO}{\mathcal{O}}
\newcommand{\dbar}{\bar\partial}
\newcommand{\half}{\tfrac12}
\newcommand{\iso}{\xrightarrow{\sim}}

\DeclareMathOperator{\Ric}{Ric}
\DeclareMathOperator{\Aut}{Aut}
\DeclareMathOperator{\Kur}{Kur}
\DeclareMathOperator{\Alb}{Alb}
\DeclareMathOperator{\Spec}{Spec}
\DeclareMathOperator{\Vol}{Vol}

\title[Non-integrable Einstein deformations]{A negative K\"ahler-Einstein threefold with non-integrable infinitesimal Einstein deformations}
\author[A. Krishna]{Ari Krishna}
\address{Department of Mathematics, Harvard University, Cambridge, Massachusetts 02138, USA}

\date{}

\begin{document}

\begin{abstract}
We construct a smooth canonically polarized threefold, not biholomorphic to a product of positive-dimensional varieties, whose normalized K\"ahler-Einstein metric admits a non-integrable infinitesimal Einstein deformation. The same tangent direction is non-integrable as an infinitesimal complex deformation. In fact, the space of infinitesimal Einstein deformations in our example has real dimension $8$, its integrable directions form a real $6$-dimensional subspace, and every direction outside that subspace is obstructed. This answers both parts of a suitably generalized version of a question posed by Dai, Wang, and Wei in real dimension $6$.
\end{abstract}

\maketitle

\section{Introduction}

Let $(M,g)$ be a compact Einstein manifold. An \emph{infinitesimal Einstein deformation} (IED) is a transverse-traceless symmetric $2$-tensor in the kernel of the linearized Einstein operator. It is \emph{integrable} if it is tangent, modulo diffeomorphisms and homotheties, to a smooth curve of Einstein metrics. If $(M,g,J)$ is K\"ahler-Einstein, one may also consider \emph{infinitesimal complex deformations} (ICDs), namely tangent directions to the space of complex structures. These are parametrized by $H^1(M,T_M)$, and their integrability is governed by the Kuranishi space.

Schwahn and Semmelmann \cite[\S6.3]{SchwahnSemmelmann} formulate the following version of a question of Dai, Wang, and Wei \cite{DaiWangWei}.

\begin{question}[Dai-Wang-Wei; Schwahn-Semmelmann]\label{q:DWW}
Does there exist a compact K\"ahler-Einstein manifold $(M,g,J)$ of negative scalar curvature, with $\dim_{\R}M\neq4$, admitting a non-integrable infinitesimal complex deformation or a non-integrable infinitesimal Einstein deformation?
\end{question}

When the Einstein constant is negative, Koiso gives a real-linear isomorphism between $H^1(M,T_M)$ and the space of IEDs \cite{Koiso}. An obstructed ICD, however, could conceivably integrate through Einstein metrics that cease to be K\"ahler, and holonomy rigidity for nonzero K\"ahler-Einstein metrics is not known in full generality \cite[\S6.3]{SchwahnSemmelmann}. Moreover, every IED of a negative K\"ahler-Einstein metric is integrable to second Einstein order \cite[Thm.~5.3]{NagySemmelmann}. Thus, a quadratic complex obstruction is not a second-order Einstein obstruction.

Nagy's recent computation of the third-order Einstein equation \cite{NagyThird} implies that a nonzero primary Kodaira-Spencer obstruction prevents the corresponding IED from extending to third Einstein order, even along a curve of non-K\"ahler-Einstein metrics. Combining this with an equivariant construction from a rigid but not infinitesimally rigid surface yields the following result whose proof we establish in this paper.

\begin{theorem}\label{thm:main}
There exist a smooth projective threefold $X$, a class $\alpha_X\in H^1(X,T_X)$, and a K\"ahler-Einstein metric $g$ on $X$ with $\Ric_g=-g$ such that:
\begin{enumerate}[label=\textup{(\roman*)},leftmargin=2.3em]
\item $K_X$ is ample, and $X$ is not biholomorphic to a product of positive-dimensional varieties;
\item $[\alpha_X,\alpha_X]\neq0$ in $H^2(X,T_X)$, so $\alpha_X$ is a non-integrable ICD;
\item the IED corresponding to $\alpha_X$ under Koiso's isomorphism is non-integrable, and is obstructed at third Einstein order.
\end{enumerate}
More precisely, for a smooth curve $B$ of genus $2$ there is a decomposition
$$H^1(X,T_X)=\C\alpha_X\oplus H^1(B,T_B),$$
and the ICD and IED associated with $\lambda\alpha_X+\beta$ are integrable if and only if $\lambda=0$. Hence, the real $8$-dimensional space of IEDs has a real $6$-dimensional subspace consisting precisely of its integrable directions.
\end{theorem}

The threefold has an \'etale cover that is a product, but is not itself a product. The obstruction is inherited from the surface factor, while the free action on the curve makes the diagonal quotient smooth. 

\textbf{Acknowledgments.} We thank Tristan Ozuch, Dhruv Goel, and Laasya Nagumalli for helpful conversations and proofreading. 

\section{Constructing the threefold}

We write $q(S)=h^1(S,\cO_S)$.

\begin{proposition}\label{prop:surface}
There exist a smooth projective surface $S$, a faithful action of $G\cong(\Z/7)^4$
on $S$, and a class $\alpha\in H^1(S,T_S)^G$ such that
$$K_S\ \text{is ample},\qquad q(S)=0,\qquad
H^1(S,T_S)=\C\alpha,$$
and
$$
[\alpha,\alpha]\neq0\quad\text{in }H^2(S,T_S).$$
\end{proposition}

\begin{proof}
B\"ohning, Graf von Bothmer, and Pignatelli construct a smooth surface $S$ with
$$\Kur(S)\cong\Spec\C[x]/(x^2),$$
$K_S$ ample, and $q(S)=0$ \cite[Thm.~1.1 and Thms.~6.3, 6.5]{BVP}. Their surface is a totally ramified abelian cover whose deck group is $G\cong(\Z/7)^4$ \cite[Cor.~5.4]{BVP}. The natural transformation of \cite[Lem.~2.11]{BVP} identifies the tangent space of the branch-data deformation functor with $H^1(S,T_S)^G$. The verification in the proof of \cite[Thm.~5.5]{BVP}, using \cite[Thm.~3.19]{BVP}, of the hypotheses of \cite[Cor.~2.13]{BVP} shows that this transformation is an isomorphism. Hence,
$$H^1(S,T_S)=H^1(S,T_S)^G.$$

The aforementioned description of $\Kur(S)$ gives $\dim_{\C}H^1(S,T_S)=1$; let $\alpha$ be a generator. Since $K_S$ is ample, $H^0(S,T_S)=0$, so the deformation functor is prorepresentable. The first-order deformation associated with $\alpha$ is induced by
$$\C[x]/(x^2)\longrightarrow\C[t]/(t^2),
\qquad x\longmapsto at,$$
for some $a\neq0$. A lift to $\C[t]/(t^3)$ would have to send $x$ to $at+bt^2$, but
$$(at+bt^2)^2=a^2t^2\neq0\pmod{t^3}.$$
It follows that $\alpha$ has no second-order lift. By the primary obstruction criterion recalled in Lemma~\ref{lem:primary}, this is equivalent to $[\alpha,\alpha]\neq0$.
\end{proof}

Let $B$ be a smooth projective curve of genus $2$. Its fundamental group surjects onto $G$: in the standard presentation, send the four generators of $\pi_1(G)$ to a basis of the abelian group $G$; the surface-group relation maps to zero. By the Riemann existence theorem, the resulting connected topological cover carries a unique structure of a finite \'etale Galois cover
$$q\colon C\longrightarrow B$$
with Galois group $G$. Then, Riemann-Hurwitz gives that the genus of $C$ is 
$$g(C)=1+|G|\bigl(g(B)-1\bigr)=1+7^4=2402.$$
Note that the diagonal action of $G$ on $S\times C$ is free. With this in hand, let
$$X=(S\times C)/G$$
and denote the quotient map by $\pi\colon S\times C\to X$.

We turn to examining the deformation theory of this quotient. Let $p_S$ and $p_C$ be the two projections.

\begin{lemma}[Descent and splitting]\label{lem:descent}
There are natural isomorphisms
$$H^1(X,T_X)\cong H^1(S\times C,T_{S\times C})^G
\cong H^1(S,T_S)^G\oplus H^1(C,T_C)^G.
$$
Moreover, \'etale descent along $q$ identifies
$$
H^1(C,T_C)^G\cong H^1(B,T_B).
$$
If $\alpha_X\in H^1(X,T_X)$ is defined by
$$
\pi^*\alpha_X=p_S^*\alpha,
$$
then, for $\lambda\in\C$ and $\beta\in H^1(B,T_B)$,
\begin{equation}\label{eq:bracket-splitting}
\pi^*\bigl[\lambda\alpha_X+\beta,\lambda\alpha_X+\beta\bigr]
=\lambda^2p_S^*[\alpha,\alpha].
\end{equation}
In particular, the bracket on the left-hand side is nonzero if and only if $\lambda\neq0$.
\end{lemma}

\begin{proof}
Pullback identifies the Kodaira-Spencer differential graded Lie algebra of $X$ with the invariant sub-DGLA on $S\times C$. Since the group-invariants functor under a finite group is exact in characteristic zero, we get
$$H^i(X,T_X)\cong H^i(S\times C,T_{S\times C})^G.$$
Now,
$$T_{S\times C}=p_S^*T_S\oplus p_C^*T_C.$$
The K\"unneth formula, in conjunction with $H^0(S,T_S)=H^0(C,T_C)=0$, yields
$$H^1(S\times C,T_{S\times C})
=H^1(S,T_S)\oplus H^1(C,T_C).$$
The asserted identification for $C\to B$ arises from the same \'etale descent argument.

Let $\widetilde\beta\in H^1(C,T_C)^G$ be the pullback of $\beta$. Classes pulled back from different factors have zero Kodaira-Spencer bracket. Moreover,
$$[\widetilde\beta,\widetilde\beta]=0,$$
because it is represented by a $(0,2)$-form on a curve. This proves \eqref{eq:bracket-splitting}. Finally, the K\"unneth decomposition
\begin{align*}
H^2(S\times C,T_{S\times C})={}&H^2(S,T_S)
\oplus\bigl(H^1(S,T_S)\otimes H^1(C,\cO_C)\bigr)\\
&\oplus\bigl(H^1(C,T_C)\otimes H^1(S,\cO_S)\bigr)
\end{align*}
shows that $p_S^*\colon H^2(S,T_S)\to H^2(S\times C,T_{S\times C})$ is injective. Since $[\alpha,\alpha]$ is $G$-invariant, it remains nonzero after invariant descent. The last assertion follows.
\end{proof}

\section{From complex obstructions to Einstein obstructions}

We first recall the primary complex obstruction. For a compact complex manifold $Z$, its Kodaira-Spencer DGLA is
$$
L_Z=\bigl(A^{0,\bullet}(Z,T_Z),\dbar,[\ ,\ ]\bigr).$$
It governs the Kuranishi deformation functor via the Maurer-Cartan equation
$$\dbar\phi+\half[\phi,\phi]=0;$$
consult \cite{Kodaira,Kuranishi,ManettiDGLA,Sernesi}for details.

\begin{lemma}[Primary obstruction]\label{lem:primary}
Let $Z$ be a compact complex manifold and $\xi\in H^1(Z,T_Z)$. The class $\xi$ admits a second-order lift if and only if
$$[\xi,\xi]=0\quad\text{in }H^2(Z,T_Z).$$
In particular, a class with nonzero bracket is not integrable. The construction can be shown to be functorial under morphisms of Kodaira-Spencer DGLAs.
\end{lemma}

\begin{proof}
Choose a $\dbar$-closed representative, which we will continue to denote by $\xi$. A second-order Maurer-Cartan lift has the form
$$\phi(t)=t\xi+t^2\phi_2\pmod{t^3}.$$
The coefficient of $t^2$ in the Maurer-Cartan equation is
$$\dbar\phi_2+\half[\xi,\xi]=0.$$
Such a $\phi_2$ exists exactly when the cohomology class of $[\xi,\xi]$ vanishes. Functoriality follows from compatibility with the differential and bracket.
\end{proof}

Next, if $(M,g,J)$ is compact K\"ahler-Einstein with negative Einstein constant, Koiso's results give a real-linear isomorphism
\begin{equation}\label{eq:koiso}
\kappa_g\colon H^1(M,T_M)_{\R}\iso\varepsilon(g),
\end{equation}
where $\varepsilon(g)$ denotes the real vector space of IEDs \cite[Lem.~6.6, Props.~7.3, 8.2-8.3, and Lem.~9.3]{Koiso}. Here, the subscript means that the complex vector space is regarded as a real one; the correspondence is unchanged by a constant rescaling of the metric.

\begin{proposition}\label{prop:bridge}
Let $(M,g,J)$ be a compact K\"ahler-Einstein manifold with $\Ric_g=Eg$ and $E<0$. If $\xi\in H^1(M,T_M)$ satisfies
$$[\xi,\xi]\neq0\quad\text{in }H^2(M,T_M),$$
then the IED $\kappa_g(\xi)$ is not integrable. More precisely, it cannot be extended to third order as an Einstein deformation.
\end{proposition}

\begin{proof}
Write $h=\kappa_g(\xi)$ and suppose, more generally, that $h$ extends to an Einstein deformation through third order. All the normalizations and equations below depend only on the $3$-jet \cite[Thm.~5.12]{NagyThird}; in particular, they apply to the jet of every actual Einstein curve. Rescale the jet to have the same total volume as $g$. Since $h$ is trace-free,
$$\left.\frac{d}{dt}\right|_{t=0}\Vol(g_t)
=\half\int_M\operatorname{tr}_g(h)\,dV_g=0,$$
so this rescaling does not change the tangent vector. The Einstein constant is then $E$ to the required order. Indeed, on the fixed-volume hypersurface , the Einstein-Hilbert functional is stationary at every Einstein metric, whereas its value at $g_t$ is $(\dim M)E_t\Vol(g)$.

We use $g$ to identify symmetric $2$-tensors with self-adjoint endomorphisms. Since $h$ is already divergence-free, the normalizing diffeomorphism in the proof of \cite[Prop.~3.2]{NagyThird} may be chosen with trivial first jet. Thus, the normalization in \cite[Thm.~1.2(i)]{NagyThird} preserves $h$, and we may write
$$g^{-1}g_t=\operatorname{id}+th+\frac{t^2}{2}h_2+\frac{t^3}{6}h_3+O(t^4).$$
Nagy's third-order Einstein equation \cite[Thm.~1.2(ii)(b)]{NagyThird} implies
\begin{equation}\label{eq:nagy}
\mathcal B(h_2-h^2)+[h,h]^c=0.
\end{equation}
Here, $\mathcal B$ is the real form of the Dolbeault operator and $[\ ,\ ]^c$ is the real Kodaira-Spencer bracket. Consequently, the cohomology class of $[h,h]^c$ vanishes. Nagy identifies the complexification of this real class with the primary Dolbeault obstruction in $H^{0,2}(M,T_M)$; under Koiso's correspondence, it is a nonzero universal scalar multiple of $[\xi,\xi]$ \cite[\S2.2 and the discussion following Thm.~1.2]{NagyThird}. Hence, $[\xi,\xi]=0$, a contradiction.
\end{proof}

\section{Proof of the main theorem}

We now prove all the assertions of Theorem~\ref{thm:main} in one fell swoop.

\subsection{The metric and its infinitesimal deformations}

The diagonal action on $S\times C$ is free, so $X$ is smooth and projective and $\pi$ is finite \'etale. Moreover,
$$\pi^*K_X=K_{S\times C}=p_S^*K_S\otimes p_C^*K_C$$
is ample. Ampleness descends under a finite surjection, so $K_X$ is ample.

Let $g_S$ and $g_C$ be the unique K\"ahler-Einstein metrics normalized by
$$\Ric_{g_S}=-g_S,
\qquad
\Ric_{g_C}=-g_C;$$
their existence is guaranteed by the Aubin-Yau theorem \cite{Aubin,Yau}. Every element of $G$ is a holomorphic automorphism of $S$ or $C$, so uniqueness makes it an isometry of the corresponding normalized metric. Hence, the product metric $g_S\oplus g_C$ is $G$-invariant, and therefore descends to a metric $g$ on $X$. Since
$$\Ric_{g_S\oplus g_C}=-(g_S\oplus g_C),$$
the descended metric satisfies $\Ric_g=-g$. It is the unique normalized K\"ahler-Einstein metric on $X$.

Proposition~\ref{prop:surface} and Lemma~\ref{lem:descent} give
$$H^1(X,T_X)=\C\alpha_X\oplus H^1(B,T_B).$$
Since $g(B)=2$,
$$\dim_{\C}H^1(B,T_B)=3,
\qquad
\dim_{\C}H^1(X,T_X)=4.$$
Koiso's isomorphism \eqref{eq:koiso} therefore gives
$$\dim_{\R}\varepsilon(g)=8.$$

Let $\xi=\lambda\alpha_X+\beta$. If $\lambda\neq0$, Lemma~\ref{lem:descent} gives
$$[\xi,\xi]\neq0.$$
Lemma~\ref{lem:primary} shows that $\xi$ is a non-integrable ICD, and Proposition~\ref{prop:bridge} shows that $\kappa_g(\xi)$ is a non-integrable IED, obstructed at third Einstein order.

It remains to show that every direction with $\lambda=0$ is integrable in both senses. Every $\beta\in H^1(B,T_B)$ is tangent to a curve of complex structures $B_t$, because the deformation space of a smooth curve is unobstructed. Now keep the underlying topological $G$-cover $C\to B$ fixed, and pull back the complex structures on $B_t$; this produces a curve of \'etale holomorphic $G$-covers $C_t\to B_t$. The quotients
$$X_t=(S\times C_t)/G$$
form a curve of canonically polarized complex manifolds tangent to $\beta$. Thus, $\beta$ is an integrable ICD. By the parametric Aubin-Yau theorem, or equivalently Koiso's deformation theorem in the negative case, the normalized K\"ahler-Einstein metrics on $X_t$ vary smoothly \cite[Prop.~10.1]{Koiso}. Their tangent can be prescribed to represent $\kappa_g(\beta)$, so this IED is integrable, as was sought.

\subsection{Non-productness}

It remains to show that $X$ is not biholomorphic to a product. Though Schwann-Semmelmann's survey does not necessarily mandate this, we remark that the problem is only interesting when this condition is imposed.

To this end, first note that the map $q\circ p_C\colon S\times C\to B$ descends to a smooth isotrivial fibration
$$f\colon X\longrightarrow B$$
whose fibers are isomorphic to $S$. Its monodromy is the composite
$$\pi_1(B)\twoheadrightarrow G\hookrightarrow\Aut(S).$$
The first map is surjective and the second is faithful, so this monodromy is nontrivial. Since $K_S$ is ample, $\Aut(S)$ is finite; thus, the monodromy is well-defined up to conjugacy, and a product fibration would have trivial monodromy.

Since $q(S)=0$, invariant descent and K\"unneth give
$$H^0(X,\Omega_X^1)
=H^0(S\times C,\Omega^1_{S\times C})^G
=H^0(C,\Omega_C^1)^G
=H^0(B,\Omega_B^1).$$
Consequently, the Albanese map $a_X\colon X\to\Alb(X)$ is constant on the fibers of $f$ and factors as
$$a_X=j\circ f$$
for a nonconstant map $j\colon B\to\Alb(X)$. The map $j$ is finite onto its image. Since $f$ has connected fibers, it is the Stein factorization of $a_X$. In particular, the Albanese image of $X$ is one-dimensional.

Suppose that $X\cong A\times D$ with both factors positive-dimensional. Any decomposition with more factors may of course be grouped in this form. Restricting $K_X$ to the fibers of the two projections shows that $K_A$ and $K_D$ are ample, so both factors are projective. The Albanese image of a product is the product of the Albanese images. Since the Albanese image of $X$ is a curve, after interchanging the factors we may assume that $q(A)=0$ and that the Albanese map depends only on $D$. For fixed $d\in D$, the connected set $f(A\times\{d\})$ lies in the finite fiber of $j$ over the point $a_X(A\times\{d\})$, and hence it is a point. Thus, $f$ factors through the projection to $D$: there is a holomorphic map $h\colon D\to B$ such that
$$f=h\circ p_D.$$

If $\dim D=1$, then $h$ is finite. Connectedness of the fibers of $f$ forces $h$ to have degree $1$, so $h$ is an isomorphism. The fibration $f$ is then a product projection $A\times B\to B$ and has trivial monodromy, contradicting the monodromy computed above. If $\dim D=2$, then $A$ is a curve with $q(A)=0$, hence, $A\cong\PP^1$, contradicting ampleness of $K_A$. The remaining possibility would make one factor a point. Thus, $X$ has no nontrivial product decomposition.

This completes the proof of Theorem~\ref{thm:main}.

\end{document}